\documentclass[11pt]{amsart}

\usepackage{amsmath,amssymb,amsthm}

\newtheorem{theorem}{Theorem}[section]
\newtheorem{proposition}[theorem]{Proposition}
\newtheorem{lemma}[theorem]{Lemma}

\theoremstyle{definition}

\newcommand{\E}{\mathbb E}

\newcommand{\Pcal}{\mathcal P}
\newcommand{\Ecal}{\mathcal E}
\newcommand{\one}{\mathbf 1}

\begin{document}

\title{A note on large clusters of $E_2$-numbers}

\author{Genheng Zhao}
\address{Academy of Mathematics and Systems Science, Chinese Academy of Sciences,\\
55 Zhongguancun East Road, Haidian District, Beijing 100190, China}
\email{zhaogenheng@amss.ac.cn}


\begin{abstract}
Let $q_n$ denote the $n$th product of two distinct primes.  By completing a
square in Sono's sieve and solving the resulting one-dimensional variational
problem, we prove unconditionally that, for every $\varepsilon>0$ and all
sufficiently large integers $\rho$,
\[
 \liminf_{n\to\infty}(q_{n+\rho}-q_n)
 \leq \exp\left((\pi+\varepsilon)\sqrt\rho\right).
\]
\end{abstract}

\maketitle

\section{Introduction}

An $E_2$-number is a product of two distinct primes.  Write
$q_1<q_2<\cdots$ for the $E_2$-numbers and set
\[
 H_\rho(E_2)=\liminf_{n\to\infty}(q_{n+\rho}-q_n).
\]
Here and below $\rho$ is a positive integer.  The sequence of $E_2$-numbers
is a natural almost-prime setting for sieve methods.  Goldston, Graham,
Pintz, and Y\i ld\i r\i m applied the GPY sieve to this sequence, proving an
unconditional bound of $6$ for consecutive gaps and a general large-cluster
estimate \cite{GGPY}.  Maynard's multidimensional refinement of the GPY
sieve \cite{Maynard} was subsequently adapted by Sono to a sifted
$E_2$-minorant \cite{Sono}.  Sono proved the general estimate
\[
 H_\rho(E_2)\leq
 \exp\left(\frac{(2+\varepsilon)\rho}{3\theta\log\rho}\right)
\]
under level-$\theta$ distribution hypotheses \cite[Theorem 1.1]{Sono}, and
an observation of Neshime recorded in \cite[Remark 7.1]{Sono} gives
\[
 H_\rho(E_2)\ll
 \sqrt\rho\exp\left(\sqrt{\frac{8\rho}{\theta}}\right).
\]
Thus the latter has exponential constant $4$ when $\theta=\frac12$.
Liu, Park, and Song extended related bounded-gap results to products of a
general number of distinct primes \cite{LPS}; Goldston, Panidapu, and
Schettler later obtained explicit bounds for small fixed clusters within
Sono's framework \cite{GPS}.  We show that the unconditional exponential
constant $4$ above may be replaced by $\pi$.

\begin{theorem}\label{thm:main}
Let $0<\theta\leq1$, and assume
$\mathrm{BV}[\theta,\Pcal]$ and $\mathrm{BV}[\theta,E_2]$ in the sense used by
Sono.  For every $\varepsilon>0$ and all sufficiently large integers $\rho$,
\[
 H_\rho(E_2)
 \leq
 \exp\left(\left(\frac{\pi}{\sqrt{2\theta}}+\varepsilon\right)
 \sqrt\rho\right).
\]
In particular, using the level $\theta=\frac12$ with an arbitrarily small
loss, the unconditional consequence is
\[
 H_\rho(E_2)\leq
 \exp\left((\pi+\varepsilon)\sqrt\rho\right).
\]
\end{theorem}

The proof has three ingredients.  First, Sono's $J-L-M$ functional is
rewritten as an exact positive square plus a nonnegative $J$-term.  Second, a
logarithmically reparametrized product weight reduces the resulting lower
bound to a one-dimensional problem depending on
\[
 r=f'(0),\quad E=\int_0^1 f'(y)^2\,dy,\quad
 S=\int_0^1 f(y)^2\,dy.
\]
Third, the sharp inequality
\[
 E(1-S)\geq\frac{\pi^2}{16}
 \quad\text{when}\quad f(0)=0, f(1)=1, S\leq\frac12,
\]
determines the optimal constant within this product family.  A boundary-layer
construction makes the strict concentration condition $E<r^2$ compatible
with the sharp value.  The remaining sections justify the smooth
approximation and the order of all limiting operations.

\section{Notation and distribution hypotheses}

An $E_2$-number in this paper is an integer $p_1p_2$ with $p_1$ and $p_2$
distinct primes.  Let $\Pcal$ denote the set of primes, let $\mu$ be the
M\"obius function, and let $\varphi$ be Euler's totient function.  The first
few $E_2$-numbers are $6,10,14,15,21,22,\ldots$.  For a set $\mathcal A$,
$\one_{\mathcal A}$ denotes its indicator function.  Sono does
not insert the characteristic function of all $E_2$-numbers directly into
the sieve.  For fixed $0<\eta<\frac14$ and $N$ large, he uses the sifted
minorant
\[
 \beta_{N,\eta}(n)=
 \begin{cases}
 1,&n=p_1p_2,\quad N^\eta<p_1\leq N^{\frac12}<p_2,\\
 0,&\text{otherwise}.
 \end{cases}
\]
Every integer counted by $\beta_{N,\eta}$ is an $E_2$-number, so positivity
for this minorant is sufficient for a gap result for the full sequence.
The restriction on the smaller prime is also the source of the lower
$\xi$-cutoff in the $L$ and $M$ integrals.  This is precisely the definition
in \cite[(1.10)]{Sono}.

For $(a,q)=1$, put
\begin{align*}
 \pi_\beta(N;q,a)&=
 \sum_{\substack{N<n\leq2N\\n\equiv a\pmod q}}\beta_{N,\eta}(n),\\
 \pi_{\beta,q}(N)&=
 \sum_{\substack{N<n\leq2N\\(n,q)=1}}\beta_{N,\eta}(n).
\end{align*}
For the prime sequence, put
\begin{align*}
 \pi^\flat(N)&=\#\{p\text{ prime}:N\leq p<2N\},\\
 \pi^\flat(N;q,a)&=
 \#\{p\text{ prime}:N\leq p<2N,\ p\equiv a\pmod q\}.
\end{align*}
Sono's hypothesis $\mathrm{BV}[\theta,E_2]$ says that, for fixed $\eta$, every
$\varepsilon>0$, and every $B>0$,
\begin{equation}\label{eq:BV-E2}
 \sum_{q\leq N^{\theta-\varepsilon}}\mu(q)^2
 \max_{(a,q)=1}
 \left|\pi_\beta(N;q,a)-\frac{\pi_{\beta,q}(N)}{\varphi(q)}\right|
 \ll_{B,\eta}\frac{N}{(\log N)^B}.
\end{equation}
The prime hypothesis $\mathrm{BV}[\theta,\Pcal]$ is the statement that, for
every $\varepsilon>0$ and $B>0$,
\begin{equation}\label{eq:BV-prime}
 \sum_{q\leq N^{\theta-\varepsilon}}\mu(q)^2
 \max_{(a,q)=1}
 \left|\pi^\flat(N;q,a)-\frac{\pi^\flat(N)}{\varphi(q)}\right|
 \ll_B\frac{N}{(\log N)^B}.
\end{equation}
These are exactly the conventions in \cite[(1.11) and (1.12)]{Sono}.
Notice that Sono uses $N<n\leq2N$ in the definition of the
$\beta$-counts, but $N\leq n<2N$ in the sieve sum below.  We retain those
conventions; changing either endpoint affects at most endpoint terms and has
no bearing on the asymptotics.  The ordinary Bombieri--Vinogradov theorem and
Motohashi's convolution theorem \cite{Motohashi} supply these hypotheses at level
$\theta=\frac12$ in Sono's convention.  In applications below, $k$ and hence
$\eta=\frac{\theta}{2k^2}$ are fixed before $N\to\infty$; the dependence on
$\eta$ in \eqref{eq:BV-E2} is therefore harmless.

\section{Sono's positivity functional}

Put
\[
 \mathcal R_k=\left\{\boldsymbol u\in[0,\infty)^k:
 \sum_{i=1}^k u_i\leq1\right\},
 \quad I_k(F)=\int_{\mathcal R_k}F(\boldsymbol u)^2\,d\boldsymbol u.
\]
Let $\mathcal H=\{h_1,\ldots,h_k\}$ be admissible.  Thus, for every prime
$p$, the residues $h_1,\ldots,h_k$ do not cover all residue classes modulo
$p$.  Set
\[
 D_0=\log\log\log N,\quad W=\prod_{p\leq D_0}p,
\]
and choose $\nu_0$ so that $(\nu_0+h_m,W)=1$ for every $m$.  For a fixed
$\delta>0$, put $R=N^{\frac{\theta}{2}-\delta}$.  A smooth function $F$
supported on $\mathcal R_k$ determines Sono's Maynard--Selberg coefficients
and the nonnegative weight
\[
 w_n=\left(\sum_{d_m\mid n+h_m\ (1\leq m\leq k)}
 \lambda_{d_1,\ldots,d_k}\right)^2,
\]
where $\lambda_{d_1,\ldots,d_k}$ is defined in \cite[(2.1)]{Sono}; this is
the weight in \cite[(2.2)]{Sono}.  With
$\beta=\beta_{N,\eta}$, set
\begin{equation}\label{eq:S-def}
 S(N,\rho)=
 \sum_{\substack{N\leq n<2N\\n\equiv\nu_0\pmod W}}
 \left(\sum_{m=1}^k\beta(n+h_m)-\rho\right)w_n.
\end{equation}
Because $w_n\geq0$, positivity of \eqref{eq:S-def} implies that some
translate $n+\mathcal H$ contains at least $\rho+1$ $E_2$-numbers.

For $1\leq m\leq k$, write $\boldsymbol u_{\widehat m}$ for all coordinates
except $u_m$ and define
\[
 A_m(\boldsymbol u_{\widehat m})
 =\int_0^1F(\boldsymbol u)\,du_m,
 \quad
 J_k^{(m)}(F)=\|A_m\|_2^2.
\]
Here $\langle\cdot,\cdot\rangle$ and $\|\cdot\|_2$ refer to the inner product
and norm on $L^2([0,1]^{k-1},d\boldsymbol u_{\widehat m})$.  Extend $F$ by
zero outside $\mathcal R_k$.  Sono's affine transform is
\[
 F_{m,0}(\boldsymbol u;\xi)=F\left(u_1,\ldots,u_{m-1},
 \frac{2\xi}{\theta}+\frac{\theta-2\xi}{\theta}u_m,
 u_{m+1},\ldots,u_k\right),
 \quad 0<\xi<\frac{\theta}{2}.
\]
Put
\[
 B_{m,\xi}(\boldsymbol u_{\widehat m})
 =\int_0^1F_{m,0}(\boldsymbol u;\xi)\,du_m.
\]
The affine formula is exactly \cite[(3.42)]{Sono}; its mixed inner product is
\cite[(3.43)]{Sono}, and the corresponding square norm is
\cite[(5.20)]{Sono}.  Thus, in Sono's integrated notation,
\[
 L_{k,0}^{(m)}(F)
 =\int_\eta^{\frac{\theta}{2}}
 \frac{\frac{\theta}{2}-\xi}{1-\xi}
 \langle A_m,B_{m,\xi}\rangle\,\frac{d\xi}{\xi}
\]
as in \cite[Proposition 3.7]{Sono}, and
\[
 M_{k,0}^{(m)}(F)
 =\int_\eta^{\frac{\theta}{2}}
 \frac{(\frac{\theta}{2}-\xi)^2}{1-\xi}
 \|B_{m,\xi}\|_2^2\,\frac{d\xi}{\xi}.
\]
This is the integrated quantity in \cite[(5.21) and (5.22)]{Sono}.
The $E_2$ positivity expression in \cite[(6.3)]{Sono} is
\begin{align}\label{eq:sono-functional}
 \Ecal_{k,\eta}(F)
 ={}&-\theta\sum_{m=1}^kL_{k,0}^{(m)}(F)
 +\frac{\theta^2}{4}\log\frac{1-\eta}{\eta}
   \sum_{m=1}^kJ_k^{(m)}(F)\notag\\
 &+\sum_{m=1}^kM_{k,0}^{(m)}(F).
\end{align}

The part of Sono's sieve used in this paper may now be stated as follows.

\begin{proposition}[Sono's positivity criterion]\label{prop:sono-criterion}
Assume $\mathrm{BV}[\theta,\Pcal]$ and
$\mathrm{BV}[\theta,E_2]$, and fix
\[
 0<\eta<\min\left\{\frac14,\frac{\theta}{2}\right\}.
\]
Let $F$ be smooth and supported on $\mathcal R_k$.  If
\begin{equation}\label{eq:sono-positive}
 \Ecal_{k,\eta}(F)-\frac{\rho\theta}{2}I_k(F)>0,
\end{equation}
then, after choosing the auxiliary support loss $\delta>0$ sufficiently
small, one has $S(N,\rho)>0$ for all sufficiently large $N$.  Consequently,
\[
 H_\rho(E_2)\leq\operatorname{diam}(\mathcal H),
 \quad
 \operatorname{diam}(\mathcal H)=\max_{1\leq i<j\leq k}|h_i-h_j|.
\]
\end{proposition}

\begin{proof}
This is the $E_2$ specialization of Sono's asymptotic evaluation
\cite[(6.3)]{Sono}, followed by the observation after
\cite[(2.3)]{Sono}.  In Sono's notation one first has
$\theta'=\theta-2\delta$; continuity of the leading coefficient at
$\delta=0$ preserves the strict inequality \eqref{eq:sono-positive} for all
sufficiently small fixed $\delta>0$.  The asymptotic is then taken with
$k,\eta,F,$ and $\delta$ fixed as $N\to\infty$.
\end{proof}

\begin{lemma}[Exact square completion]\label{lem:square}
For every $F$ for which the displayed integrals exist,
\begin{align}
 \Ecal_{k,\eta}(F)
 ={}&\sum_{m=1}^k\int_\eta^{\frac{\theta}{2}}
 \frac{(\frac{\theta}{2}-\xi)^2}{\xi(1-\xi)}
 \left\|B_{m,\xi}-\frac{\theta}{\theta-2\xi}A_m\right\|_2^2
 \,d\xi\notag\\
 &+\frac{\theta^2}{4}\log\frac{2-\theta}{\theta}
 \sum_{m=1}^kJ_k^{(m)}(F).\label{eq:square}
\end{align}
In particular, the right side is nonnegative when $0<\theta\leq1$.
\end{lemma}

\begin{proof}
For fixed $m$ and $\xi$, expansion of the square gives
\begin{align*}
 &\frac{(\frac{\theta}{2}-\xi)^2}{1-\xi}
 \left\|B_{m,\xi}-\frac{\theta}{\theta-2\xi}A_m\right\|_2^2\\
 &\quad=
 \frac{(\frac{\theta}{2}-\xi)^2}{1-\xi}\|B_{m,\xi}\|_2^2
 -\theta\frac{\frac{\theta}{2}-\xi}{1-\xi}
 \langle A_m,B_{m,\xi}\rangle
 +\frac{\theta^2}{4(1-\xi)}\|A_m\|_2^2.
\end{align*}
It remains to use
\[
 \int_\eta^{\frac{\theta}{2}}\frac{d\xi}{\xi(1-\xi)}
 =\log\frac{1-\eta}{\eta}+\log\frac{\theta}{2-\theta}.
\]
Moving the second logarithm to the other side proves
\eqref{eq:square}.
\end{proof}

\section{The general product ansatz}

Fix $f\in C^2[0,1]$ with
\[
 f(0)=0,\quad f(1)=1,\quad r=f'(0)\neq0.
\]
For $A>1$, set
\[
 T_A=\frac{e^A-1}{A},\quad
 y_A(s)=\frac{\log(1+As)}{A},\quad 0\leq s\leq T_A,
\]
and define
\[
 H_A(s)=f(y_A(s)),\quad
 g_A(s)=H_A'(s)=\frac{f'(y_A(s))}{1+As}.
\]
Extend $H_A$ by $H_A(s)=1$ for $s\geq T_A$ and $g_A$ by zero.
Then
\[
 \int_0^{T_A}g_A(s)\,ds=1.
\]
Let
\[
 E=\int_0^1f'(y)^2\,dy,\quad
 S=\int_0^1f(y)^2\,dy,\quad
 \gamma_A=\int_0^{T_A}g_A(s)^2\,ds.
\]

\begin{lemma}[One-dimensional asymptotics]\label{lem:one-d-asymptotics}
As $A\to\infty$,
\begin{align}
 \gamma_A&=\int_0^1f'(y)^2e^{-Ay}\,dy
 =\frac{r^2}{A}+O_f\left(\frac{1}{A^2}\right),\label{eq:gamma}\\
 Q_A&:=\int_0^{T_A}\frac{H_A(s)^2}{s}\,ds
 =A\int_0^1\frac{f(y)^2}{1-e^{-Ay}}\,dy
 =AS+O_f(1).\label{eq:Q}
\end{align}
If $X_A$ has density $\frac{g_A(s)^2}{\gamma_A}$ on $[0,T_A]$, then the exact identity
\begin{equation}\label{eq:mean}
 \E X_A=\frac{E-\gamma_A}{A\gamma_A}
\end{equation}
holds.  Consequently $\E X_A\to \frac{E}{r^2}$.
\end{lemma}

\begin{proof}
The substitution $1+As=e^{Ay}$ gives \eqref{eq:gamma}.  Watson's
endpoint estimate, or one integration by parts after subtracting $r^2$,
gives its second equality.  The same substitution gives the exact integral
in \eqref{eq:Q}.  Since $f(y)=O_f(y)$ at zero,
\[
 \int_0^1 f(y)^2\frac{e^{-Ay}}{1-e^{-Ay}}\,dy=O_f\left(\frac{1}{A^3}\right),
\]
so in fact the error in \eqref{eq:Q} is $O_f(A^{-2})$; the weaker displayed
bound is sufficient.  Finally,
\begin{align*}
 A\int_0^{T_A}s g_A(s)^2\,ds
 &=\int_0^1f'(y)^2(1-e^{-Ay})\,dy\\
 &=E-\gamma_A,
\end{align*}
which proves \eqref{eq:mean}.
\end{proof}

For $k\geq2$ define the sharply truncated product
\begin{equation}\label{eq:F-product}
 F_{k,A}(\boldsymbol u)=
 \prod_{i=1}^k g_A(ku_i)\one_{\mathcal R_k}(\boldsymbol u).
\end{equation}

\begin{lemma}[Simplex concentration]\label{lem:concentration}
Suppose $E<r^2$ and $A=c\log k$ with fixed $0<c\leq1$.  Then
\[
 I_k(F_{k,A})=(1+o_f(1))\frac{\gamma_A^k}{k^k}.
\]
For every $m$, the same estimate with $k-1$ factors holds on
\[
 \mathcal G_m=\left\{\boldsymbol u_{\widehat m}:
 \sum_{i\neq m}u_i\leq1-\frac{T_A}{k}\right\}.
\]
The error is uniform in $m$.
\end{lemma}

\begin{proof}
Let $X_{A,1},\ldots,X_{A,k}$ be independent copies of $X_A$.  Under their
product probability measure the condition defining $\mathcal R_k$ is
$X_{A,1}+\cdots+X_{A,k}\leq k$.  By
Lemma~\ref{lem:one-d-asymptotics}, there is a constant $\delta_f>0$ such
that the mean of this sum is at most $(1-\delta_f)k$ for
all sufficiently large $k$.  Moreover
\[
 \operatorname{Var}(X_A)\leq \E X_A^2\leq T_A\E X_A\ll_f T_A.
\]
Chebyshev's inequality therefore bounds the discarded mass by
$O_f(\frac{T_A}{k})=o(1)$, because $\frac{T_A}{k}\ll\frac{1}{\log k}$ when $c=1$ and is smaller
when $c<1$.  Replacing $k$ by $k-1$ and the threshold by $k-T_A$ leaves a
positive linear gap, since $T_A=o(k)$.  This proves both assertions.
\end{proof}

On $\mathcal G_m$ the simplex cutoff does not constrain $u_m$.  With
$s=\frac{2k\xi}{\theta}$, a direct affine change of variable gives
\begin{equation}\label{eq:good-region-identity}
 B_{m,\xi}-\frac{\theta}{\theta-2\xi}A_m
 =-\frac{\theta}{\theta-2\xi}\frac{H_A(s)}{k}
 \prod_{i\neq m}g_A(ku_i).
\end{equation}
This identity remains valid for $s>T_A$ because $H_A(s)=1$ and the
transformed $g_A$-integral then vanishes.

\begin{lemma}[Smooth approximation of the Sono functionals]
\label{lem:smoothing}
Fix $k$, $\eta>0$, and $0<\theta\leq1$.  If
$F_j\to F$ in $L^2(\mathcal R_k)$ and the functions are extended by zero to
$[0,1]^k$, then
\[
 I_k(F_j)\to I_k(F),\quad
 J_k^{(m)}(F_j)\to J_k^{(m)}(F),
\]
and likewise
\[
 L_{k,0}^{(m)}(F_j)\to L_{k,0}^{(m)}(F),\quad
 M_{k,0}^{(m)}(F_j)\to M_{k,0}^{(m)}(F).
\]
Consequently every bounded $L^2$ function supported on $\mathcal R_k$ may,
for the purpose of a strict Sono positivity inequality, be replaced by a
function in $C_c^\infty(\mathcal R_k^\circ)$.
\end{lemma}

\begin{proof}
The assertions for $I_k$ and $J_k^{(m)}$ follow from Cauchy--Schwarz and
the fact that fiber integration from $L^2([0,1]^k)$ to
$L^2([0,1]^{k-1})$ has norm at most one.  For the transformed fiber put
\[
 b_\xi=\frac{\theta-2\xi}{\theta}.
\]
After changing variables in the $m$th coordinate,
\[
 B_{m,\xi}(F)=\frac{1}{b_\xi}
 \int_{\frac{2\xi}{\theta}}^1
 F(\boldsymbol u_{\widehat m},v)\,dv.
\]
Cauchy--Schwarz therefore gives
\begin{equation}\label{eq:B-operator}
 \|B_{m,\xi}(F)\|_2
 \leq b_\xi^{-\frac12}\|F\|_2.
\end{equation}
Near $\xi=\frac{\theta}{2}$, the weight in $L$ is $O(b_\xi)$ and the
weight in $M$ is $O(b_\xi^2)$.  Thus \eqref{eq:B-operator} supplies the
integrable majorants $O(b_\xi^{\frac12})$ and $O(b_\xi)$, respectively.
The lower endpoint causes no difficulty because $\eta>0$.  Dominated
convergence, polarization for $L$, and the identity
$|x^2-y^2|\leq|x-y|(|x|+|y|)$ for $M$ prove the claimed continuity.
Finally, $C_c^\infty(\mathcal R_k^\circ)$ is dense in
$L^2(\mathcal R_k)$.
\end{proof}

\begin{proposition}[Product-weight transfer]\label{prop:transfer}
Let $f$ satisfy the hypotheses above and $E<r^2$.  Put $A=c\log k$ with
$0<c\leq1$ and choose $\eta=\frac{\theta}{2k^2}$.  Then the smooth approximants to
\eqref{eq:F-product} allowed in Sono's sieve may be chosen so that
\begin{equation}\label{eq:transfer}
 \frac{\Ecal_{k,\eta}(F)}{(\frac{\theta}{2})I_k(F)}
 \geq
 \frac{\theta A}{2r^2}
 \left(\log k-A(1-S)+O_f(\log A)+o_f(\log k)\right).
\end{equation}
The $o_f(\log k)$ term is taken with $f$ fixed.
\end{proposition}

\begin{proof}
Insert \eqref{eq:good-region-identity} into
Lemma~\ref{lem:square}, restrict the nonnegative square to
$\mathcal G_m$, and use Lemma~\ref{lem:concentration}.  The change of
variables $s=\frac{2k\xi}{\theta}$ gives, after summing over $m$,
\begin{align*}
 \frac{\Ecal_{k,\eta}(F)}{I_k(F)}
 \geq(1-o_f(1))\frac{\theta^2}{4\gamma_A}
 \left\{
 \int_{\frac{1}{k}}^{k}\frac{H_A(s)^2}{s(1-\frac{\theta s}{2k})}\,ds
 +\log\frac{2-\theta}{\theta}
 \right\}.
\end{align*}
Split the integral at $T_A$.  On $[T_A,k]$, $H_A=1$; exact integration
and combination with the last logarithm gives
\[
 \log\frac{2k}{\theta T_A}+O\left(\frac{T_A}{k}\right).
\]
On $[\frac{1}{k},T_A]$, replacing $(1-\frac{\theta s}{2k})^{-1}$ by $1$ costs
$O_f(1)$ after normalization, and replacing the lower endpoint by zero costs
$O_f(k^{-2})$.  Lemma~\ref{lem:one-d-asymptotics} now yields
\[
 \frac{\Ecal_{k,\eta}(F)}{I_k(F)}
 \geq
 \frac{\theta^2A}{4r^2}
 \left(\log k-A+AS+O_f(\log A)+o_f(\log k)\right),
\]
which is \eqref{eq:transfer}.

Lemma~\ref{lem:smoothing} now replaces the two sharp cutoffs by smooth
functions supported strictly inside $\mathcal R_k$ while preserving the
strict lower bound to any prescribed accuracy.  No uniformity in $k$ is
needed: the approximant is fixed before $N\to\infty$.  With this order of
choices, the remainders in \cite[(3.45) and (5.22)]{Sono} tend to zero even
though their constants may depend on the derivative norms of the fixed
approximant.
\end{proof}

\section{Optimization and the one-dimensional variational problem}

\subsection{Optimization in the scale parameter}

Dividing \eqref{eq:transfer} out and taking $A=c\log k$ gives
\begin{equation}\label{eq:ratio-main}
 \frac{\Ecal_{k,\eta}(F)}{(\frac{\theta}{2})I_k(F)}
 \geq
 \left(\frac{\theta}{2r^2}c\{1-c(1-S)\}+o_f(1)\right)(\log k)^2.
\end{equation}
The elementary optimization is
\begin{equation}\label{eq:c-opt}
 \max_{0<c\leq1}c\{1-c(1-S)\}
 =
 \begin{cases}
 \displaystyle\frac{1}{4(1-S)},&S\leq\frac12,\\[6pt]
 S,&S\geq\frac12.
 \end{cases}
\end{equation}
At $S=\frac12$ both formulas agree.

\subsection{The sharp variational constant}

We require the following elementary variational statement.

\begin{lemma}[Sharp variational inequality]\label{lem:variational}
Let $f\in H^1(0,1)$, $f(0)=0$, $f(1)=1$, and
$S=\int_0^1f^2$.  If $S\leq\frac12$, then
\begin{equation}\label{eq:sharp-var}
 \left(\int_0^1f'(x)^2\,dx\right)(1-S)
 \geq\frac{\pi^2}{16}.
\end{equation}
The constant is sharp.  If $S\geq\frac12$, then
\begin{equation}\label{eq:DN}
 \int_0^1f'(x)^2\,dx\geq\frac{\pi^2}{4}S.
\end{equation}
\end{lemma}

\begin{proof}
The second assertion is the Dirichlet--Neumann Poincar\'e inequality on
$[0,1]$; no condition at $1$ is required.

For the first assertion, Cauchy's inequality gives $E:=\int f'^2\geq1$.
Thus \eqref{eq:sharp-var} is immediate when $S\leq1-\frac{\pi^2}{16}$; in
particular it is immediate for $S\leq\frac13$.  For fixed
$S\in[\frac13,\frac12]$, the direct method gives a minimizer of $E$ subject to the
two endpoint conditions and the $L^2$ constraint.  The Euler--Lagrange
equation and the endpoint data give
\[
 f_t(x)=\frac{\sin(tx)}{\sin t},\quad 0\leq t\leq\frac{\pi}{2}.
\]
Indeed, replacing a minimizer by its absolute value does not change either
constraint or the Dirichlet energy, so a minimizing Euler--Lagrange solution
may be taken nonnegative.  It therefore lies on the first Sturm branch.  The
hyperbolic branch corresponds to $S<\frac13$ and has already been covered.
On the trigonometric first branch, $f_t(x)$ increases pointwise with $t$ for
$0<x<1$: this follows because $z\cot z$ is strictly decreasing.  Thus
$S(t)$ increases from $\frac13$ to $\frac12$ as $t$ runs from $0$ to
$\frac{\pi}{2}$, which justifies the stated parameter range.
For the trigonometric branch,
\begin{align*}
 S(t)&=\frac{\frac12-\frac{\sin(2t)}{4t}}{\sin^2t},\\
 E(t)&=\frac{t^2\left(\frac12+\frac{\sin(2t)}{4t}\right)}{\sin^2t}.
\end{align*}
A differentiation shows that $E(t)(1-S(t))$ decreases on
$[0,\frac{\pi}{2}]$.  More precisely,
\begin{align*}
 \frac{d}{dt}\{E(t)(1-S(t))\}
 =\frac{\cos^2t}{2\sin^5t}
 \left(2t^2\cos t-t\sin t-\sin^2t\cos t\right).
\end{align*}
The required sign is therefore equivalent to
\[
 t\tan t+\sin^2t>2t^2\quad(0<t<\frac{\pi}{2}),
\]
and this follows cleanly from Huygens' inequality
\[
 2\sin t+\tan t>3t.
\]
Indeed, if $a=\frac{\sin t}{t}$ and $b=\frac{\tan t}{t}$, Huygens gives
$2a+b>3$, whence
\[
 b+a^2=(b+2a)+(a-1)^2-1>2.
\]
For completeness, Huygens' inequality itself follows by differentiating:
with $c=\cos t$,
\[
 \frac{d}{dt}(2\sin t+\tan t-3t)
 =\frac{(c-1)^2(2c+1)}{c^2}>0,
\]
while the expression vanishes at $t=0$.  Hence
\[
 E(t)(1-S(t))\geq
 E\left(\frac{\pi}{2}\right)
 \left(1-S\left(\frac{\pi}{2}\right)\right)
 =\frac{\pi^2}{16}.
\]
Equality is attained by $f(x)=\sin(\frac{\pi x}{2})$, proving sharpness.
\end{proof}

Combining Lemma~\ref{lem:variational} with $E<r^2$ gives the same constant
in both cases of \eqref{eq:c-opt}.  If $S\leq\frac12$, then
\[
 r^2(1-S)>E(1-S)\geq\frac{\pi^2}{16}.
\]
If $S\geq\frac12$, then $\frac{r^2}{S}>\frac{E}{S}\geq\frac{\pi^2}{4}$.  It follows from
\eqref{eq:ratio-main} that the coefficient of $(\log k)^2$ cannot exceed
$\frac{2\theta}{\pi^2}$ within this product family.  The next subsection shows that it
can be approached.

\subsection{Boundary-layer recovery of the endpoint slope}

The extremizer $\phi(x)=\sin(\frac{\pi x}{2})$ has
$\phi'(0)^2=\frac{\pi^2}{4}$, whereas
$\int_0^1\phi'^2=\frac{\pi^2}{8}$; thus it does not directly force the concentration
condition $E<r^2$ at the desired scale $r^2\downarrow E$.  A boundary layer
repairs this mismatch.

\begin{lemma}[Boundary-layer approximation]\label{lem:boundary}
For every $\varepsilon>0$ there exists $f\in C^2[0,1]$ such that
\[
 f(0)=0,\quad f(1)=1,\quad
 \int_0^1f'^2<f'(0)^2,\quad
 \int_0^1f^2<\frac12,
\]
and
\[
 f'(0)^2\left(1-\int_0^1f^2\right)
 <\frac{\pi^2}{16}+\varepsilon.
\]
The function can be chosen smooth after an arbitrarily small perturbation.
\end{lemma}

\begin{proof}
Let
\[
 \phi_t(x)=\frac{\sin(tx)}{\sin t},\quad t<\frac{\pi}{2}.
\]
Set
\[
 E_t=\int_0^1\phi_t'(x)^2\,dx,
 \quad S_t=\int_0^1\phi_t(x)^2\,dx.
\]
Then $S_t\to\frac12$ and $E_t(1-S_t)\to\frac{\pi^2}{16}$ as
$t\uparrow\frac{\pi}{2}$.  Fix $t$ close to $\frac{\pi}{2}$ and choose $r_0$ with
$r_0^2>E_t$ and $r_0^2-E_t$ arbitrarily small.  On $[0,\delta]$, take the
cubic Hermite polynomial with value $0$ and derivative $r_0$ at $0$, and
with value $\phi_t(\delta)$ and derivative $\phi_t'(\delta)$ at $\delta$.
On $[\delta,1]$, retain $\phi_t$.

The rescaled Hermite formula shows that its first derivative is bounded by a
constant depending only on $r_0$ and $t$, uniformly for small $\delta$.
Consequently the boundary layer contributes $O_{r_0,t}(\delta)$ to both
$\int f'^2$ and $\int f^2$, and
\[
 \int_0^1f'^2\to E_t,\quad \int_0^1f^2\to S_t
 \quad(\delta\to0).
\]
For small $\delta$ the strict inequality $\int f'^2<r_0^2$ holds.  Now send
$\delta\to0$ while retaining $\int f^2<\frac12$, then send
$r_0^2\downarrow E_t$, and finally
$t\uparrow\frac{\pi}{2}$.  To obtain a smooth function, mollify only in a
shrinking neighborhood of the joining point $x=\delta$, leaving fixed
neighborhoods of $0$ and $1$ unchanged.  This preserves $f'(0)=r_0$ and the
two endpoint values, while convergence in $H^1$ preserves all three strict
inequalities for a sufficiently small mollification scale.
\end{proof}

\section{Proof of the main theorem}

We now prove Theorem~\ref{thm:main} using Proposition~\ref{prop:transfer}.  By
Lemma~\ref{lem:boundary}, choose and then fix $f$ so close to the sharp
constant that the optimized coefficient in \eqref{eq:ratio-main} is greater
than
\[
 \frac{2\theta}{\pi^2}-\delta_1
\]
for a prescribed $\delta_1>0$.  Next choose $k$ so that
\[
 \log k=\left(\frac{\pi}{\sqrt{2\theta}}+\frac{\varepsilon}{3}\right)
 \sqrt\rho+O(1).
\]
For sufficiently large $\rho$, and then sufficiently small $\delta_1$, the
right side of \eqref{eq:ratio-main} exceeds $\rho$.  Thus Sono's coefficient
\[
 \Ecal_{k,\eta}(F)-\frac{\rho\theta}{2}I_k(F)
\]
is positive.

For each fixed $k$, Sono's asymptotic as $N\to\infty$ then makes
$S(N,\rho)>0$.  Therefore infinitely many translates of the chosen
admissible $k$-tuple contain at least $\rho+1$ sifted $E_2$-numbers
\cite[(2.3) and (6.3)]{Sono}.  For completeness, take $h_1,\ldots,h_k$ to be
the first $k$ primes exceeding $k$.  If $p\leq k$, the residue class $0$
is omitted modulo $p$; if $p>k$, at most $k<p$ residue classes are occupied.
Thus the tuple is admissible, and the prime number theorem gives
$\operatorname{diam}(\mathcal H)=O(k\log k)$.  Hence
\[
 H_\rho(E_2)\ll k\log k.
\]
Since $\log(k\log k)=\log k+o(\sqrt\rho)$, the extra factor is absorbed by
the remaining $\frac{2\varepsilon}{3}$, proving the first assertion of
Theorem~\ref{thm:main}.

The classical Bombieri--Vinogradov theorem for primes and Motohashi's
corresponding input for the relevant $E_2$ convolution give every fixed
$\theta<\frac12$ in Sono's formulation.  Letting $\theta\uparrow\frac12$ after
fixing $\varepsilon$ yields the asserted constant $\pi$.

The order of choices in this argument is
\[
 \varepsilon\longrightarrow f\longrightarrow\rho
 \longrightarrow k,A,\eta
 \longrightarrow F_{\mathrm{smooth}},\delta
 \longrightarrow N\to\infty.
\]
In particular, $\eta$ and the smooth approximant are fixed when Sono's
finite-$N$ asymptotic is applied.  This completes the proof of
Theorem~\ref{thm:main}.

\end{document}